\documentclass[12pt]{amsart}

\usepackage{amssymb}

\usepackage{enumitem}

\usepackage{graphicx}

\makeatletter
\@namedef{subjclassname@2010}{%
  \textup{2010} Mathematics Subject Classification}
\makeatother

\usepackage[T1]{fontenc}

\newtheorem{theorem}{Theorem}[section]

\newtheorem{proposition}[theorem]{Proposition}

\theoremstyle{definition}
\newtheorem{definition}[theorem]{Definition}

\numberwithin{equation}{section}

\begin{document}


\baselineskip=17pt


\title[The Generalized  Grand Wiener Amalgam Spaces and the boundedness of Hardy-Littlewood maximal operators]{The Generalized  Grand Wiener Amalgam Spaces and the boundedness of Hardy-Littlewood maximal operators}

\author[A.Turan Gurkanli]{A.Turan Gurkanli}
\address{Istanbul Arel University Faculty of Science and Letters\\ Department
of Mathematics and Computer Sciences\\
\'Istanbul\\
Turkey}
\email{turangurkanli@arel.edu.tr}

\date{}

\begin{abstract}
In \cite{g5}, we defined and investigated the grand Wiener amalgam space  $W(L^{p),\theta_1}(\Omega), L^{q),\theta_2}(\Omega))$ , where $1<p,q<\infty, \theta_1>0, \theta_2>0$, $\Omega\subset\mathbb R^{n} $ and the Lebesgue measure of $\Omega$ is finite.  In the present paper we generalize this space and define the generalized grand Wiener amalgam space $W(L_{a}^{p)}(\mathbb R^{n}), L_{b}^{q)}(\mathbb R^{n})),$ where $L_{a}^{p)}(\mathbb R^{n})$ and $L_{b}^{q)}(\mathbb R^{n}),$ are the generalized grand Lebesgue spaces, (see \cite{u}, \cite{su3}). Later we  investigate some basic properties. Next  we study embeddings for these spaces and we discuss boundedness and unboundedness of the Hardy-Littlewood maximal operator between some generalized grand Wiener amalgam spaces.
 
\end{abstract}
\maketitle



\maketitle

\section{Preliminaries}

Let $\le p\le \infty$ and $\Omega\subseteq R^n $ be an open subset. We denote by  $\mid A\mid$  the Lebesgue measure of a measurable set $A\subset R^n$ .  The translation and modulation operators are given by%
\begin{align*}
T_{x}f\left( t\right) =f\left( t-x\right) ,\text{ }M_{\xi }f\left( t\right)
=e^{ i<\xi, t>}f\left( t\right) ,\text{ }t,\text{ }x,\text{ }\xi \in\mathbb R^{n}.
\end{align*} 
 Let $\omega$  be a real-valued function defined on $\mathbb R^{n}.$  If $\omega$ is  positive, measurable and   locally integrable function  vanishing on a set of zero measure, it is called a weight function.  We define the weighted space $L^p(\Omega,\omega)$ with the norm
\begin{align}
\left\Vert {f}\right\Vert_{{L^p}(\Omega,\omega)}=\left( \int_{\Omega }\left\vert f\right\vert ^{p }\omega(x)dx\right) ^{\frac{1}{p }},1< p<\infty,
\end{align}
and 
\begin{align*}
\|f\|_{{L^{\infty}}(\Omega,\omega)}=esssup_{x\in\Omega}{|f(x)|\omega(x)}, 
\end{align*}
 (see  \cite {cr} \cite {fg},\cite {rgl}). A weight function $\omega$ is called \textit{ submultiplicative}, if
\begin{align*}
 \omega(x+y)\leq\omega(x)\omega(y),\ \ \forall x,y\in \mathbb R^{n}.
\end{align*}
  A weight function $\omega$ is called \textit{Beurling's weight function} on $\mathbb R^n$ if submultiplicative and $\omega(x)\ge1, $ \cite{w}. 
The weighted space  $L^p(\Omega,\omega)$ is called solid space, if $g\in L^{p}{(\Omega,\omega}), f\in L^1_{loc}(\Omega,\omega)$ and $|f(x)|\leq|g(x)|$ l.a.e, implies $f\in L^{p}{(\Omega,\omega)} $ and $\|f\|_{{{L^p}(\Omega,\omega)}}\leq\|g\|_{{{L^p}(\Omega,\omega)}}.$

 Let $|\Omega|<\infty.$
The grand Lebesgue space $L^{p)}\left( \Omega \right) $ was introduced by Iwaniec-Sbordone in \cite{is}. This Banach space is defined by the norm
\begin{align}
\left\Vert {f}\right\Vert _{p)}=\sup_{0<\varepsilon \leq p-1}\left(
\varepsilon \int_{\Omega }\left\vert f\right\vert ^{p-\varepsilon }dx
\right) ^{\frac{1}{p-\varepsilon }}
\end{align}
where $1<p<\infty . $ For $0<\varepsilon \leq p-1,$ $L^{p}\left( \Omega \right) \subset L^{p)}\left( \Omega \right) \subset \L^{p-\varepsilon }\left( \Omega \right)$
 hold. For some properties and applications of $L^{p)}\left( \Omega \right) $  we refer to papers \cite{gis} and \cite{g4}. An application to amalgam spaces we refer to paper \cite{g5}. 
 Sometimes in definition of grand  Lebesgue space a parameter $\theta>0$ is added with the change of the factor $\varepsilon$ to $\varepsilon^{\theta},  $ \cite{gis}. We will consider $\theta=1,$ since further the parameter will not play much importance.
  Also the subspace ${C_{0}^{\infty }}$ is not dense in $L^{p)}\left( \Omega \right),$ where  ${C_{0}^{\infty }}$ is the space of infinitely differentiable complex valued functions with compact support. Its closure consists of functions $f \in L^{p)}\left(\Omega \right)$  such that
\begin{equation}
\lim_{\varepsilon \rightarrow 0}\varepsilon ^{\frac{\theta }{p-\varepsilon }
}\left\Vert f\right\Vert _{p-\varepsilon}=0, 
\end{equation}
 \cite{cr}, \cite{gis}.
 It is also known that the grand Lebesgue space  $L^{p),\theta }\left( \Omega \right) ,$ is not reflexive.

In all above mentioned studies only sets $\Omega$ of finite measure were allowed, based on the embedding
\begin{equation*}
L^{p}\left( \Omega \right) \hookrightarrow
L^{p-\varepsilon }\left( \Omega \right) .
\end{equation*}

Let $1< p< \infty$ and $\Omega\subseteq R^n $ be an open subset. We define the generalized grand Lebesgue space $L_a^{p)}(\Omega)$ on a set $\Omega$ of possibly infinite measure as follows (see \cite{su3} and  \cite{u}): 
\begin{align}
\notag L_a^{p)}(\Omega)&=\{ f:\left\Vert {f}\right\Vert _{L^{p)}_a(\Omega)} = \sup_{0<\varepsilon \leq p-1}\varepsilon \left( \int_{\Omega }\left\vert f\right\vert ^{p-\varepsilon} a(x)^{\frac{\varepsilon}p} dx
\right) ^{\frac{1}{p-\varepsilon}
}\\&=\sup_{0<\varepsilon \leq p-1}\varepsilon \left\Vert {f}\right\Vert_{L^{p-\varepsilon }(\Omega, a^{\frac{\varepsilon}p})}<\infty \}.
\end{align}
The norm of this space is equivalent to the norm
\begin{align}
\notag\left\Vert {f}\right\Vert _{L^{p)}_a(\Omega)}& = \sup_{0<\varepsilon \leq p-1}\left(
\varepsilon \int_{\Omega }\left\vert f\right\vert ^{p-\varepsilon} a(x)^{\varepsilon} dx
\right) ^{\frac{1}{p-\varepsilon}
}\\&=\sup_{0<\varepsilon \leq p-1}\varepsilon ^{\frac{1 }{p-\varepsilon }}\left\Vert {f}\right\Vert
_{L^{p-\varepsilon }(\Omega, a^{\varepsilon})}.
\end{align}
 We call $a(x)$ the grandizer of the space $L_a^{p)}(\Omega)$. It is known that $L_a^{p)}(\Omega)$ is a Banach space   \cite{u}.
If $\Omega$ is bounded and $a(x)\equiv1$ then there holds the embedding $$L^{p}( \Omega)\hookrightarrow L_a^{p)}(\Omega).$$ 
\section{Generalized Grand Wiener Amalgam Spaces and some of its basic properties}
Let $1< p< \infty$ and $\Omega\subseteq \mathbb R^{n} $ be an open subset. The space $(L_a^{p)}(\Omega))_{loc}$ 
consists
of (classes of) measurable functions $f$ $:\Omega \rightarrow \mathbb{C}$
such that $f\chi _{K}\in L_a^{p)}(\Omega),$ for any compact subset $K\subset \Omega ,$
 where $\chi _{K}$ is the characteristic function of $K.$ 
It is known by Lemma 3.1 in \cite{su3} and Lemma 3 in \cite{u} that the embedding $$L^{p}( \Omega)\hookrightarrow L_a^{p)}(\Omega).$$ holds if and only if $a\in L^{1}( \Omega).$ This implies that  if $a\in L^{1}( \Omega)$ then $ L^{p}(\Omega) _{loc}\hookrightarrow L_a^{p)}(\Omega)_{loc}.$ 

Now  in the spirit of \cite{fe},\cite{g5},\cite{h} we will give the definition of grand Wiener amalgam space.
\begin{definition}
Let $1< p,q<\infty$ and $ a(x),b(x)$  be weight functions on $\mathbb R^{n}$.  Fix a compact $ Q\subset\mathbb R^{n}$ with nonempty interior. The generalized grand Wiener amalgam space $W(L_{a}^{p)}(\mathbb R^{n}), L_{b}^{q)}\mathbb R^{n}))$ consists of all functions ( classes of ) $f\in( L_{a}^{p)}(\mathbb R^{n}))_{loc} $ such that the control function
\begin{align*}
F^{p)}_{f,u}(x) &=\| f.\chi _{Q+x}\| _{L_{a}^{p)}(\mathbb R^{n})} \\
&=\sup_{0<\varepsilon \leq p-1}\varepsilon \left( \int_{\mathbb R^{n}}\left\vert f(t)\chi _{Q+x}(t)\right\vert ^{p-\varepsilon}a(t)^{\frac{\varepsilon}p}dt \right) ^{\frac{1}{p-\varepsilon }} \\
&=\sup_{0<\varepsilon \leq p-1}\varepsilon\left\Vert { f.\chi _{Q+x}}\right\Vert
_{L^{p-\varepsilon }(\mathbb R^{n}, a^{\frac{\varepsilon}p})}
\end{align*}
lies in $ L_{b}^{q)}(\mathbb R^{n}).$ The norm on $W(L_{a}^{p)}(\mathbb R^{n}), L_{b}^{q)}(\mathbb R^{n})),$ (or shortly $W(L_{a}^{p)}, L_{b}^{q)}),$  is
\begin{equation}
\|f\|_{W(L_{a}^{p)}(\mathbb R^{n}), L_{b}^{q)}(\mathbb R^{n}))}=\|F^{p)}_f\|_{ L_{b}^{q)}(\mathbb R^{n})}=\|\| f.\chi _{Q+x}\| _{(L_{a}^{p)}(\mathbb R^{n})}\|_{ L_{b}^{q)}(\mathbb R^{n}).}
\end{equation}
\end{definition}
\begin{proposition}
 The generalized grand Wiener amalgam space  $W(L_{a}^{p)}(\mathbb R^{n}), L_{b}^{q)}(\mathbb R^{n}))$  is translation and modulation invariant.
\end{proposition}
\begin{proof}
 Let $ f\in W(L_{a}^{p)}(\mathbb R^{n}), L_{b}^{q)}\mathbb R^{n})).$ Then $ f\chi_{Q+x}\in L_{a}^{p)} (\mathbb R^{n})$ and  $F^{p)}_f(x) =\| f.\chi _{Q+x}\| _{L_{a}^{p)}(\mathbb R^{n})}\in  L_{b}^{q)}(\mathbb R^{n})).$ It is known from Proposition 1 in $[18]$  that  $L_{a}^{p)}(\mathbb R^{n})$ is translation invariant. Thus we have $Tyf\in L_{a}^{p)}(\mathbb R^{n})$ and
\begin{align}
F^{p)}_{T_yf}(x) =\|( T_{y}f).\chi _{Q+x}\| _{L_{a}^{p)}(\mathbb R^{n})}=\| f.\chi _{Q+x-y}\| _{L_{a}^{p)}(\mathbb R^{n})}=(T_{y}F^{p)}_{f})(x) 
\end{align}
 for $y\in \mathbb R^{n}$. Again since  $ L_{b}^{q)}(\mathbb R^{n})$ is translation invariant,  from (3.1) and (3.2), we find
\begin{align*}
\|T_yf\|_{W(L_{a}^{p)}(\mathbb R^{n}), L_{b}^{q)}(\mathbb R^{n}))}=\|(T_{y}F^{p)}_{f})(x)\|_{ L_{b}^{q)}(\mathbb R^{n})}<\infty,
\end{align*}
and so $T_yf\in W(L_{a}^{p)}(\mathbb R^{n}), L_{b}^{q)}(\mathbb R^{n})).$ That means  $ W(L_{a}^{p)}(\mathbb R^{n}), L_{b}^{q)}\mathbb R^{n}))$ is translation invariant.

  Now let $ f\in W(L_{a}^{p)}(\mathbb R^{n}), L_{a}^{q)}(\mathbb R^{n})),$  and $ \text{ }\xi \in\mathbb R^{n}$.
Then
\begin{align}
\notag F^{p)}_{M_\xi f}(x) &=\|({M_\xi f}) \chi _{Q+x}\| _{(L_{a}^{p)}(\mathbb R^{n})} \\
\notag&=\sup_{0<\varepsilon \leq p-1}\varepsilon \left( \int_{\mathbb R^{n} }\left\vert e^{i\xi t}f(t)\chi _{Q+x}(t)\right\vert ^{p-\varepsilon}a(x)^{\frac{\varepsilon}p}dt \right) ^{\frac{1}{p-\varepsilon }} \\
&=\sup_{0<\varepsilon \leq p-1}\varepsilon\left\Vert { f.\chi _{Q+x}}\right\Vert
_{L^{p-\varepsilon }(\mathbb R^{n}, a^{\frac{\varepsilon}p})}=\| f \chi _{Q+x}\| _{L_{a}^{p)}(\mathbb R^{n})}.
\end{align}
By (3.3) we find 
\begin{align*}
\|M_\xi f\|_{W(L_{a}^{p)}(\mathbb R^{n}), L_{a}^{q)}(\mathbb R^{n}))}&=\|F^{p)}_{M_\xi f}(x)\|_{{ L_{a}^{q)}(\mathbb R^{n})}}\\&=\|\| f \chi _{Q+x}\| _{L_{a}^{p)}(\mathbb R^{n})}\|_{ L_{a}^{q)}(\mathbb R^{n})}\\&=\| f\|_{W(L_{a}^{p)}(\mathbb R^{n}), L_{a}^{q)}(\mathbb R^{n}))}<\infty.
\end{align*}
Thus $W(L_{a}^{p)}(\mathbb R^{n}), L_{b}^{q)}\mathbb R^{n}))$ is modulation invariant. 
\end{proof}
\begin{theorem}
The generalized grand Wiener amalgam space $W(L_{a}^{p)}(\mathbb R^{n}), L_{b}^{q)}(\mathbb R^{n}))$ is a Banach space, and the definition of this space is independent of the choice of $Q$, i.e., different choices of $Q$ define the same space with equivalent norms. 
\begin{proof}
 The proof of this theorem is same as the proof  Proposition 11.3.2, in  \cite{h}. and Theorem 1 in \cite{fe}.
\begin{theorem}
Let $1< p,q< \infty,$ and let $ a_k(x), a(x), $  be weight functions  on $\mathbb R^{n}$, where $k=1,2.$  Then the norm of  $W(L_{a_1}^{p)}(\mathbb R^{n}), L_{a_2}^{q)}(\mathbb R^{n}))$ satisfies the following properties, where $f,g$ and $f_{n
\text{ }}$are in $W(L_{a_1}^{p)}(\mathbb R^{n}), L_{a_2}^{q)}(\mathbb R^{n}))$ and $\lambda
\geq 0.$ 
\end{theorem}

$1.\left\Vert f\right\Vert _{W(L_{a_1}^{p)}(\mathbb R^{n}), L_{a_2}^{q)}(\mathbb R^{n}))
}\geq 0,$

$2.\left\Vert f\right\Vert _{W(L_{a_1}^{p)}(\mathbb R^{n}), L_{a_2}^{q)}(\mathbb R^{n}))}=0 $ if and only if $f=0$ a.e in $\mathbb R^{n}, $

$3.\left\Vert\lambda f\right\Vert _{W(L_{a_1}^{p)}(\mathbb R^{n}), L_{a_2}^{q)}(\mathbb R^{n}))
}=\lambda\left\Vert f\right\Vert _{W(L_{a_1}^{p)}(\mathbb R^{n}), L_{a_2}^{q)}(\mathbb R^{n}))
},$

$4.\left\Vert f+g\right\Vert _{W(L_{a_1}^{p)}(\mathbb R^{n}), L_{a_2}^{q)}(\mathbb R^{n}))
}\leq \left\Vert f\right\Vert _{W(L_{a_1}^{p)}(\mathbb R^{n}), L_{a_2}^{q)}(\mathbb R^{n}))
}+\left\Vert g\right\Vert _{W(L_{a_1}^{p)}(\mathbb R^{n}), L_{a_2}^{q)}(\mathbb R^{n}))
},$

$5.$ if $\left\vert g\right\vert \leq \left\vert f\right\vert $ a.e. in $
\mathbb R^{n} ,$ then $\left\Vert g\right\Vert _{W(L_{a_1}^{p)}(\mathbb R^{n}), L_{a_2}^{q)}(\mathbb R^{n}))
}\leq \left\Vert f\right\Vert _{W(L_{a_1}^{p)}(\mathbb R^{n}), L_{a_2}^{q)}(\mathbb R^{n}))
},$

$6.$ if $0\leq f_{n}\uparrow f$ a.e. in $\mathbb R^{n},$ then $\left\Vert
f_{n}\right\Vert _{W(L_{a_1}^{p)}(\mathbb R^{n}), L_{a_2}^{q)}(\mathbb R^{n}))
}\uparrow
\left\Vert f\right\Vert _{W(L_{a_1}^{p)}(\mathbb R^{n}), L_{a_2}^{q)}(\mathbb R^{n}))
}.$


The first four properties follow from the definition of the norm  
$\left\Vert .\right\Vert_{W(L_{a_1}^{p)}(\mathbb R^{n}), L_{a_2}^{q)}(\mathbb R^{n}))
},$ and the corresponding properties of the generalized grand Lebesgue space.

\begin{proof}
Proof of property $5.$ 

Let $g\leq f $ a.e in $\mathbb R^{n}$.  Since  $\left\vert g\right\vert \leq \left\vert f\right\vert $ a.e. in $\mathbb R^{n},$  and $L^{p-\varepsilon}(\mathbb R^{n}, a_1^{\frac{\varepsilon}p})$ is solid, then $ \|g\|_{L^{p-\varepsilon}(\mathbb R^{n}, a_1^{\frac{\varepsilon}p})}\leq \|f\|_{L^{p-\varepsilon}(\mathbb R^{n}, a_1^{\frac{\varepsilon}p})}.$ Thus we have
\begin{align}
\notag\left\Vert g\right\Vert _{L_{a_1}^{p)}(\mathbb R^{n})}&=\sup_{0<\varepsilon\leq p-1}\varepsilon \| g\|_{L^{p-\varepsilon}(\mathbb R^{n}, a_1^{\frac{\varepsilon}p})}\\& \leq\sup_{0<\varepsilon\leq p-1}\varepsilon\| f\|_{L^{p-\varepsilon}(\mathbb R^{n}, a_1^{\frac{\varepsilon}p})}=\left\Vert f\right\Vert _{L_{a_1}^{p)}(\mathbb R^{n})}.
\end{align}
By using  (3.4)  we write
\begin{align*}
 \left\Vert g\right\Vert _{W(L_{a_1}^{p)}(\mathbb R^{n}), L_{a_2}^{q)}(\mathbb R^{n}))
}&=\|F^{p)}_{ g,a_1}(x)\|_{{ L_{a_2}^{q)}(\mathbb R^{n})}}=\|\| g \chi _{Q+x}\| _{L_{a_1}^{p)}(\mathbb R^{n})}\|_{ L_{a_2}^{q)}(\mathbb R^{n})}\\&\leq\|\| f \chi _{Q+x}\| _{L_{a_1}^{p)}(\mathbb R^{n})}\|_{ L_{a_2}^{q)}(\mathbb R^{n})}=\| f\|_{W(L_{a_1}^{p)}(\mathbb R^{n}), L_{a_2}^{q)}(\mathbb R^{n}))}. 
\end{align*}
Proof of property $6.$

If $ 0\leq f_{n}\uparrow f$ a.e in $\mathbb R^{n}$ then
\begin{align}
 \sup_{n}\|f_n\|_{L_{a}^{p)}(\mathbb R^{n})}&= \sup_{n}( \sup_{0<\varepsilon\leq p-1}\varepsilon\|f_n\|_{L^{p-\varepsilon}\notag(\mathbb R^{n}, a^{\frac{\varepsilon}p})})\\&=\sup_{0<\varepsilon\leq p-1}\varepsilon( \sup_{n}\|f_n\|_{L^{p-\varepsilon}\notag(\mathbb R^{n}, a^{\frac{\varepsilon}p})})\\&=\sup_{0<\varepsilon\leq p-1}\varepsilon( \|f\|_{L^{p-\varepsilon}(\mathbb R^{n}, a^{\frac{\varepsilon}p})})=\|f\|_{L_{a}^{p)}(\mathbb R^{n}).}  
\end{align}  
Since $f_n \uparrow f$ a.e in $\mathbb R^{n},$ then$f_n \chi _{Q+x}\uparrow f \chi _{Q+x}$ in $\mathbb R^{n}.$ By (3.5) we have
\begin{align}
F^{p)}_{f_n,a_1}(x)=\|f_n\chi_{Q+x}\|_{L_{a_1}^{p)}(\mathbb R^{n})}\uparrow\|f\chi_{Q+x}\|_{L_{a_1}^{p)}(\mathbb R^{n})}=   F^{p)}_{f,a_1}(x). 
\end{align}
Thus by (3.6)
\begin{align*}
\|f_n\|_{W(L_{a_1}^{p)}(\mathbb R^{n}), L_{a_2}^{q)}(\mathbb R^{n}))}&=\|F^{p)}_{f_n,a_1}(x)\|_{ L_{a_2}^{q)}(\mathbb R^{n}))}\uparrow\|F^{p)}_{f,a_1}(x)\|_{ L_{a_2}^{q)}(\mathbb R^{n}))}\\&=\|f\|_{W(L_{a}^{p)}(\mathbb R^{n}), L_{a}^{q)}(\mathbb R^{n})).}
\end{align*}
\end{proof}

\section{Inclusions and consequences}
\begin{proposition}
 Let $a_1(x),a_2(x),b_1(x),b_2(x)$ be a Beurling's weight functions. Then 

$$ W(L_{a_1}^{p)},L_{b_1}^{q)})\left(\mathbb R^{n}\right)\subset W(L_{a_2}^{p)},L_{b_2}^{q)})\left(\mathbb R^{n}\right)$$ if and only if there exists $C>0$ such that
\begin{align*}
\|f\|_{W(L_{a_2}^{p)},L_{b_2}^{q)})\left(\mathbb R^{n}\right)}\leq C\|f\|_{W(L_{a_1}^{p)},L_{b_1}^{q)})\left(\mathbb R^{n}\right).}
\end{align*}
\end{proposition}

\begin{proof}
 Suppose  $$ W(L_{a_1}^{p)},L_{b_1}^{q)})\left(\mathbb R^{n}\right)\subset W(L_{a_2}^{p)},L_{b_2}^{q)})\left(\mathbb R^{n}\right).$$  Define the sum norm 
  \begin{align}
\||f|||=\|f\|_{W(L_{a_1}^{p)},L_{b_1}^{q)})\left(\mathbb R^{n}\right)}+\|f\|_{W(L_{a_2}^{p)},L_{b_2}^{q)})\left(\mathbb R^{n}\right)}
  \end{align}
 in $W(L_{a_1}^{p)},L_{b_1}^{q)})\left(\mathbb R^{n}\right).$  Let $\left( f_{n}\right) _{n\in\mathbb{N}}$ be a Cauchy sequence in $W(L_{a_1}^{p)},L_{b_1}^{q)})\left(\mathbb R^{n}\right),\||.\||).$  Then $\left( f_{n}\right) _{n\in \mathbb{N}}$ is a Cauchy sequence in $ W(L_{a_1}^{p)},L_{b_1}^{q)})\left(\mathbb R^{n}\right)$ and $ W(L_{a_2}^{p)},L_{b_2}^{q)})\left(\mathbb R^{n}\right)$. Hence this sequence coverges to  functions $f$ and $g$ in $W(L_{a_1}^{p)},L_{b_1}^{q)})\left(\mathbb R^{n}\right), $ and $W(L_{a_2}^{p)},L_{b_2}^{q)})\left(\mathbb R^{n}\right),$  respectively. It is easy to show that $f=g,$ and so $W(L_{a_1}^{p)},L_{b_1}^{q)})\left(\mathbb R^{n}\right),\||.\||).$   is complete. This shows that the original norm of $ W(L_{a_1}^{p)},L_{b_1}^{q)})\left(\mathbb R^{n}\right)$ and $\||.\||$ are equivalent. Thus there exist $C_1>0, C_2>0$ such that 
\begin{align}
C_2\|f\|_{W(L_{a_2}^{p)},L_{b_2}^{q)})\left(\mathbb R^{n}\right)}\leq\||f|||\leq C_1\|f\|_{W(L_{a_1}^{p)},L_{b_1}^{q)})\left(\mathbb R^{n}\right)}.
\end{align}
This implies
\begin{align}
\|f\|_{W(L_{a_2}^{p)},L_{b_2}^{q)})\left(\mathbb R^{n}\right)}\leq C\|f\|_{W(L_{a_1}^{p)},L_{b_1}^{q)})\left(\mathbb R^{n}\right).}
\end{align}
where $C=\frac{C_1}{C_2}.$
To prove of the other direction is easy.
\end{proof}
\end{proof}
\end{theorem}
\begin{theorem}
 Let $p,q>1$ and  $a(x)\in L^{1}(\mathbb R^{n})$ ,$ b(x)\in L^{1}(\mathbb R^{n}).$ Then
\\ a)
\begin{equation*}
 W(L^{p}(\mathbb R^{n}), L^{q}(\mathbb R^{n})))\hookrightarrow W(L_{a}^{p)}(\mathbb R^{n}), L_{b}^{q)}(\mathbb R^{n})).
\end{equation*} 
holds.

In addition 
\begin{equation}
\|f\|_{W(L_{a}^{p)}(\mathbb R^{n}), L_{b}^{q)}(\mathbb R^{n}))}\leq C\|f\|_{W(L^{p}(\mathbb R^{n}), L^{q}(\mathbb R^{n}))}
\end{equation}
for some $C>0.$
\\b) For an arbitrary $\varepsilon$ and $\eta$, with  $0<\varepsilon \leq p-1,$ and  $0<\eta \leq p-1,$ the embedding
\begin{equation*}
  W(L_{a}^{p)}(\mathbb R^{n}), L_{b}^{q)}(\mathbb R^{n}))\hookrightarrow W(L^{p-\varepsilon}(\mathbb R^{n}, a^{\frac{\varepsilon}p}),L^{q-{\eta}}(\mathbb R^{n}, b^{\frac{\eta}q})
\end{equation*} 
holds. 
\end{theorem}
\begin{proof}
a) Let $a\in L^{1}(\mathbb R^{n})$  and $b\in L^{1}(\mathbb R^{n}).$  By Lemma 3 in \cite{u}, and \cite{su3} 
\begin{align}
L^{p}(\mathbb R^{n})\hookrightarrow L_{a}^{p)}(\mathbb R^{n}) , L^{q}(\mathbb R^{n})\hookrightarrow L_{b}^{q)}(\Omega). 
\end{align}
This implies that 
\begin{align}
(L^{p}(\mathbb R^{n}))_{loc}\hookrightarrow( L_{a}^{p)}))_{loc} , (L^{q}(\mathbb R^{n}))_{loc}\hookrightarrow( L_{b}^{q)}(\mathbb R^{n}))_{loc}.
\end{align} 
Let $f\in W(L^{p}(\mathbb R^{n}), L^{q}(\mathbb R^{n}))). $ Then $f\in( L^{p}(\mathbb R^{n}))_{loc}$ and $F^{p}_{f,a}(x) =\| f.\chi _{Q+x}\| _{L^{p}(\mathbb R^{n})}\in L^{q}(\mathbb R^{n}) .$ Thus by $(3.5)$ and $(3.6)$ we have $f\in( L_{a}^{p)}(\mathbb R^{n}))_{loc}$  and $F^{p)}_{f,a}(x)\in  L_{b}^{q)}(\mathbb R^{n}).$ Hence $f\in W(L_{a}^{p)}(\mathbb R^{n}), L_{b}^{q)}(\mathbb R^{n})) $ and so
\begin{equation*}
W(L^{p}(\mathbb R^{n}), L^{q}(\mathbb R^{n})))\hookrightarrow W(L_{a}^{p)}(\mathbb R^{n}), L_{b}^{q)}(\mathbb R^{n})). 
\end{equation*}
By Proposition 3.1 we obtain
\begin{align*}
\|f\|_{W(L_{a}^{p)}(\mathbb R^{n}), L_{b}^{q)}(\mathbb R^{n}))}&
\leq C\|f\|_{W(L^{p}(\mathbb R^{n}), L^{q}(\mathbb R^{n})))},
\end{align*}
for some constant $C>0.$

b)
For the proof of this part take any $ f\in W(L_{a}^{p)}(\mathbb R^{n}), L_{b}^{q)}(\mathbb R^{n})).$ 
Since $L_{a}^{p)}(\mathbb R^{n})\hookrightarrow L^{p-\varepsilon}(\mathbb R^{n}, a^{\frac{\varepsilon}p}),$ and  $L_{b}^{q)}(\mathbb R^{n})\hookrightarrow L^{q-\varepsilon}(\mathbb R^{n}, b^{\frac{\varepsilon}q}),$ then
\begin{align*}
F_{f,_a}^{p-\varepsilon }\left( x\right) =\left\Vert f.\chi _{Q+x}\right\Vert
_{L^{p-\varepsilon }(\mathbb R^{n}, a^{\frac{\varepsilon}p})}\leq C(p,a)\| f.\chi _{Q+x}\| _{L_{a}^{p)}(\mathbb R^{n})}= C(p,a) F_{f}^{p)}\left( x\right). 
\end{align*}
This implies
\begin{align*}
\|f\|_{W(L^{p-\varepsilon}(\mathbb R^{n},a^{\frac{\varepsilon}p}),L^{q-\eta}(\mathbb R^{n}, b^{\frac{\eta}q}})&=\|\left\Vert f.\chi _{Q+x}\right\Vert_{L^{p-\varepsilon }(\mathbb R^{n}, a^{\frac{\varepsilon}p})}\|_{L^{q-\eta}(\mathbb R^{n}, b^{\frac{\eta}q})}\\ &\leq\|C(p,a)\| f.\chi _{Q+x}\| _{L_{a}^{p)}(\mathbb R^{n})}\|_{{L^{q-\eta }(\mathbb R^{n}, b^{\frac{\eta}q})}}\\&\leq C(p,a)C(q,b)\|\left\Vert f.\chi _{Q+x}\right\Vert_{L_{a}^{p)}(\mathbb R^{n})}\|_{L_{b}^{q)}(\mathbb R^{n})}\\&= C(p,a)C(q,b)\|f\|_{W(L_{a}^{p)}(\mathbb R^{n}), L_{b}^{q)}(\mathbb R^{n}))}. 
\end{align*} 
This completes the proof.
\end{proof}
\begin{proposition}
 Let $1\leq p_1,p_2,q<\infty$, $p_1\leq p_2.$   Then  following embeddins 
\begin{align*}
 W(L_{a}^{p_2)}(\mathbb R^{n}), L_{b}^{q)}(\mathbb R^{n})) \hookrightarrow W(L_a^{p_1)}(\mathbb 
R^{n}),L_b^{q)}(\mathbb R^{n}),
\end{align*}
hold.
\end{proposition}
\begin{proof}
 Let $f\in W(L_{a}^{p_2)}(\mathbb R^{n}), L_{b}^{q)}(\mathbb R^{n})).$ Then $F^{p_2)}_f(x) =\| f.\chi _{Q+x}\| _{L_{a}^{p_2)}(\mathbb R^{n})}\in  L_{b}^{q)}(\mathbb R^{n}).$ Since  $p_1\leq p_2,$  by the definitions of norms of the spaces $ L_{a}^{p_1}(\mathbb R^{n})$  and $ L_{a}^{p_1}(\mathbb R^{n})$ we have
\begin{align}
F^{p_1)}_{f,a}(x) =\| f.\chi _{Q+x}\| _{L_{a}^{p_1)}(\mathbb R^{n})}\leq C\| f.\chi _{Q+x}\| _{L_{a}^{p_2)}(\mathbb R^{n})}=CF^{p_2)}_{f,a}(x)
\end{align}
for some $C>0.$ By the solidness of $L_b^{q}(\mathbb R^{n})$ and (3.4),
\begin{align*}
\|f\|_{W(L_{a}^{p_1)}(\mathbb R^{n}), L_{b}^{q)}(\mathbb R^{n}))}&=\|F^{p_1)}_{f,a}(x)\|_{ L_{b}^{q)}(\mathbb R^{n})}\leq C\|F^{p_2)}_{f,a}(x)\ \|_{ L_{b}^{q)}(\mathbb R^{n})}\\&=C\|f\|_{W(L_{a}^{p_2)}(\mathbb R^{n}), L_{b}^{q)}(\mathbb R^{n})).}
\end{align*}
Hence
\begin{align*}
 W(L_{a}^{p_2)}(\mathbb R^{n}), L_{b}^{q)}(\mathbb R^{n})) \hookrightarrow W(L_a^{p_1)}(\mathbb 
R^{n}),L_b^{q)}(\mathbb R^{n}).
\end{align*}

\end{proof}
The proof of the following Proposition is as Proposition 3.5 in \cite {g5} and Theorem 11.3.3 in \cite {h}.  Therefore, we will not give a proof of this theorem.
\begin{proposition}
Let $1<p_i ,q_i<\infty,  (i=1,2,3).$ If there exist constants $C_1>0, C_2>0$ such that for all $u\in L_{a}^{p)}(\mathbb R^{n})$ and  $v\in L_{a}^{p)}(\Omega)$
\begin{align*}
\left\Vert {uv}\right\Vert _{L^{p_3)}_a(\mathbb R^{n})}\leq C_1\left\Vert {u}\right\Vert _{L^{p_1)}_a(\mathbb R^{n})}\left\Vert {v}\right\Vert _{L^{p_2)}_a(\mathbb R^{n})}
\end{align*}
and for all  $u\in L_{a}^{q)}(\mathbb R^{n})$ and  $v\in L_{a}^{q)}(\mathbb R^{n}),$
\begin{align*}
\left\Vert {uv}\right\Vert _{L^{q_3)}_a(\mathbb R^{n})}\leq C_2\left\Vert {u}\right\Vert _{L^{q_1)}_a(\mathbb R^{n})}\left\Vert {v}\right\Vert _{L^{q_2)}_a(\mathbb R^{n})}
\end{align*}
then there exists $C>0$ such that for all $f\in W(L_a^{p_1)}(\mathbb R^{n}), L_a^{q_1)}(\mathbb R^{n}))) $ and  $g\in W(L_a^{p_2)}(\mathbb R^{n}), L_a^{q_2)}(\mathbb R^{n}))), $ we have $fg\in W(L_a^{p_3)}(\mathbb R^{n}), L_a^{q_3)}(\mathbb R^{n}))) $ and 
\begin{align*}
\left\Vert {fg}\right\Vert _ {W(L_a^{p_3)}(\mathbb R^{n}), L_a^{q_3)}(\mathbb R^{n})))}\leq C\left\Vert {f}\right\Vert _{W(L_{a}^{p_1)}(\mathbb R^{n}), L_{a}^{q_1)}(\mathbb R^{n}))}\left\Vert {g}\right\Vert _{W(L_{a}^{p_2)}(\mathbb R^{n}), L_{a}^{q_2)}(\mathbb R^{n}))}.
\end{align*}
\end{proposition}

\begin{proposition}
Let $1< p< \infty$  and let $\ a(x)\in L^1(\mathbb R^{n}).$ The closure set  $\bar{C_0^{\infty}}(\mathbb R^{n})|_{W(L_{a}^{p)}, L_{a}^{q)})}$ of the set ${C_0^{\infty}}(\mathbb R^{n})$ in the space $W(L_{a}^{p)}, L_{a}^{q)})(\mathbb R^{n})$ consists of $ f\in W(L_{a}^{p)}, L_{a}^{q)})(\mathbb R^{n})$ such that 
\begin{align}
lim_{\varepsilon\to 0} \varepsilon\|f\|_{W(L^{p-\varepsilon }(\mathbb R^{n}, a^{\frac{\varepsilon}p}),L^{q-\varepsilon }(\mathbb R^{n}, a^{\frac{\varepsilon}q})}=0. 
\end{align}
\end{proposition}
\begin{proof}
 Let  $f\in\overline{C_0^{\infty}}(\mathbb R^{n})|_{W(L_{a}^{p)}, L_{a}^{q)})}.$  Then there exists a sequence $\left( f_{n}\right) \subset
 {C_0^{\infty}}(\mathbb R^{n})$ such that 
\begin{equation*}
\|f_n-f\|_{ W(L_{a}^{p)}, L_{a}^{p)})}\rightarrow 0.
\end{equation*}
Thus for given $\delta >0,$ there exists $n_{0}\in\mathbb{N} $ such that 
\begin{equation}
\|f_{n_0}-f\|_{ W(L_{a}^{p)}, L_{a}^{p)})}<\frac{\delta 
}{2}, 
\end{equation}
for all $n>n_0.$
\ Let $r=\frac{p}{p-\varepsilon},$ and $ r'=\frac{p}{\varepsilon}.$ Then $\frac{1}{r}+\frac{1}{r'}=1.$ Since
\begin{align}
  \int\limits_{\mathbb R^{n} }(\left\vert f_{n_{0}}\left(t\right) \chi _{Q+x}(t)\right\vert ^{p-\varepsilon })^{\frac{p}{p-\varepsilon }}dt = \int\limits_{\mathbb R^{n} }\left\vert f_{n_{0}}\left(t\right) \chi _{Q+x}(t)\right\vert ^{p}dt<\infty,
\end{align}
and 
\begin{align}
 \int\limits_{\mathbb R^{n} }\left\vert a^{\frac{\varepsilon}p}(t)\right\vert^{\frac{p}{\varepsilon }}dt=  \int\limits_{\mathbb R^{n} }\left\vert a\right\vert dt=\|a(t)\|_{L^1}<\infty,
\end{align}
then $\left\vert f_{n_{0}}\left(t\right) \chi _{Q+x}\right\vert ^{p-\varepsilon }\in L^{\frac{p}{p-\varepsilon}}(\mathbb R^{n})$
and $\left\vert a\right\vert^{\frac{\varepsilon}{p}} \in L^{\frac{p}{\varepsilon}}(\Omega).$ By the H\"{o}lder's inequality we write
\begin{align}
\int\limits_{\mathbb R^{n}}\left\vert f_{n_{0}}\left(t\right) \chi _{Q+x}(t)\right\vert ^{p-\varepsilon } a(t)^{\frac{\varepsilon}{p}}dt\leq&\| (f_{n_{0}}\left(t\right) \chi _{Q+x})^{p-\varepsilon}\|_{L^{\frac{p}{p-\varepsilon}}}.\|a^{\frac{\varepsilon}{p}}\|_{L^{\frac{p}{\varepsilon}}}.
\end{align}
By (3.11) and (3.13)
\begin{align}
\varepsilon \left\Vert(f_{n_{0}} \chi _{Q+x})^{p-\varepsilon}\right\Vert _{(L^{p-\varepsilon }(\mathbb R^{n}), a^{\frac{\varepsilon}p})}&=\varepsilon \left\{ \int\limits_{\mathbb R^{n} }\left\vert f_{n_{0}}\left(t\right) \chi _{Q+x}(t)\right\vert ^{p-\varepsilon } \notag a^{\frac{\varepsilon}p}(t)dt\right\} ^{\frac{1}{
p-\varepsilon }}\\&\leq\varepsilon \left\{ \| (f_{n_{0}}\left(t\right) \chi _{Q+x})^{p-\varepsilon}\|_{L^{\frac{p}{p-\varepsilon}}}.\|\notag a^{\frac{\varepsilon}{p}}\|_{L^{\frac{p}{\varepsilon}}}\l\right\} ^{\frac{1}{
p-\varepsilon }}\\&=\varepsilon \left\{ \| (f_{n_{0}}\left(t\right) \chi _{Q+x})^{p-\varepsilon}\|_{L^{\frac{p}{p-\varepsilon}}}\right\} ^{\frac{1}{p-\varepsilon }}. \left\{\|a^{\frac{\varepsilon}{p}}\|_{L^{\frac{p}{\varepsilon}}}\l\right\} ^{\frac{1}{p-\varepsilon }}.
\end{align}
From (3.11) and (3.12) we observe that 
\begin{align*}
\|a^{\frac{\varepsilon}{p}}\|_{L^{\frac{p}{\varepsilon}}}=\|a\|^{\frac{\varepsilon}{p}}_{L^{1}}.
\end{align*}
and
\begin{align}
 \left\{ \| (f_{n_{0}}\left(t\right) \chi _{Q+x})^{p-\varepsilon}\|_{L^{\frac{p}{p-\varepsilon}}}\right\} ^{\frac{1}{
p-\varepsilon }}= \| f_{n_{0}} \chi _{Q+x})\|_{L^{p}}.
\end{align}
Since $a\in L^1 (\mathbb R^{n}),$ the right hand side of  (3.14) is finite. Then  from (3.13) and (3.14) we write
\begin{align}
\notag\varepsilon\|f_{n_0}\|_{W(L^{p-\varepsilon }(\mathbb R^{n}), a^{\frac{\varepsilon}p}),L^{q-\varepsilon }(\mathbb R^{n}),  a^{\frac{\varepsilon}q})}&=\varepsilon\left\Vert \left\Vert f_{n_{0}}\chi_{Q+x}\right\Vert _{(L^{p-\varepsilon }(\mathbb R^{n}), \notag a^{\frac{\varepsilon}p})}\right\Vert _{(L^{q-\varepsilon }(\mathbb R^{n}), a^{\frac{\varepsilon}q}) }\\&\leq\varepsilon\left\Vert \| f_{n_{0}} \chi _{Q+x}\|_{L^{{p}}}\|a\|^{\frac{\varepsilon}{p(p-\varepsilon)}}_{L^1}\right\Vert \notag_{(L^{q-\varepsilon }(\mathbb R^{n}), a^{\frac{\varepsilon}q}) }\\&=\varepsilon\|a\|^{\frac{\varepsilon}{p(p-\varepsilon)}}\notag_{L^1}\left\Vert \| f_{n_{0}} \chi _{Q+x}\|_{L^{{p}}}\right\Vert _{(L^{q-\varepsilon }(\mathbb R^{n}), a^{\frac{\varepsilon}\notag q}) }\\&=\varepsilon\|a\|^{\frac{\varepsilon}{p(p-\varepsilon)}}_{L^1}\left\Vert F^p_{ f_{n_{0}}}\right\Vert _{L^{q-\varepsilon}\notag(\mathbb R^{n}, a^{\frac{\varepsilon}q}) }\\&\leq\varepsilon\|a\|^{\frac{\varepsilon}{p(p-\varepsilon)}}_{L^1}\|a\|\notag^{\frac{\varepsilon}{q(q-\varepsilon)}}_{L^1}\|F_{ f_{n_{0}}}\|_{L^{q}(\mathbb R^{n}))}\\&=\varepsilon\|a\|^{\frac{\varepsilon}{p(p-\varepsilon)}}_{L^1}\|a\|^{\frac{\varepsilon}{q(q-\varepsilon)}}_{L^1}\|f_{n_0}\|_{W(L^{p }(\mathbb R^{n})),L^{q}(\mathbb R^{n}))}.
\end{align}
If $\varepsilon\to 0,$ the right hand side of (3.15) tends to zero. Thus
\begin{align*}
lim_{\varepsilon\to 0}\varepsilon\|f_{n_0}\|_{W(L^{p-\varepsilon }(\mathbb R^{n}, a^{\frac{\varepsilon}p}),L^{q-\varepsilon }(\mathbb R^{n}, a^{\frac{\varepsilon}q})}=0. 
\end{align*}
 Hence there exists $\varepsilon_0>0$ such that when $\varepsilon<\varepsilon_0,$
 \begin{align}
 \varepsilon\|f_{n_0}\|_{W(L^{p-\varepsilon }(\mathbb R^{n}, a^{\frac{\varepsilon}p}),L^{q-\varepsilon }(\mathbb R^{n}, a^{\frac{\varepsilon}q})}<\frac{\delta}2.
 \end{align}
 Then by (3.10) and (3.17) we obtain 
  \begin{align*}
   \varepsilon\|f\|_{W(L^{p-\varepsilon }(\mathbb R^{n}, a^{\frac{\varepsilon}p}),L^{q-\varepsilon }(\mathbb R^{n}, a^{\frac{\varepsilon}q})}&\leq\varepsilon\|f_{n_0}-f\|_{W(L^{p-\varepsilon }(\mathbb R^{n}, a^{\frac{\varepsilon}p}),L^{q-\varepsilon }(\mathbb R^{n}, a^{\frac{\varepsilon}q})}\\+\varepsilon\|f_{n_0}\|_{W(L^{p-\varepsilon }(\mathbb R^{n}, a^{\frac{\varepsilon}p}),L^{q-\varepsilon }(\mathbb R^{n}, a^{\frac{\varepsilon}q})}&\leq\frac{\delta}2+\frac{\delta}2=\delta 
   \end{align*}
   when  $\varepsilon<\varepsilon_0.$ This completes the proof.
\end{proof}
\par
It is known that if the measure of  $\Omega\subset\mathbb R^{n})$ is finite and $a(x)=1, b(x)=1$  the generalized grand Lebesgue space $L_{a}^{p)}(\mathbb R^{n})$ reduce to the classical grand Lebesgue space $L^{p)}(\Omega)$ and $ L_{b}^{q)}(\mathbb R^{n})$ reduce to $ L^{q)}(\Omega)$. Thus then the generalized grand Wiener amalgam space $W(L_{a}^{p)}(\mathbb R^{n}), L_{b}^{q)}(\mathbb R^{n}))$ reduce to grand Wiener amalgam space $W(L^{p)}(\Omega),L^{q)}(\Omega))$, (see \cite{g5}). It is also known by Proposition 4.4 that ${C_0^{\infty}}(\Omega)$ is not dense in $W(L^{p)}(\Omega),L^{q)}(\Omega))$. So we can think that ${C_0^{\infty}}(\Omega)$ is not dense in the generalized grand Wiener amalgam space $W(L_{a}^{p)}(\mathbb R^{n}), L_{b}^{q)}(\mathbb R^{n}))$.

\section{The Hardy-Littlewood Maximal Operator on Generalized Grand Wiener Amalgam spaces}

For a locally integrable function $f$ on $\mathbb R^n$, we define the (centered) Hardy- Littlewood maximal function $Mf$ of $f$ by 
\begin{align*}
Mf(x)=sup_{r>0}\frac{1}{\mid B_r(x)\mid}\int_{B_r(x)}\mid f(y)\mid dy,
\end{align*}
where $ B_r(x)$ is the open ball centered at $x.$ The supremum is taken over all balls  $B_r(x).$
\begin{proposition}
 Let $1\leq p\leq q\leq r<\infty$. If   $a\in L^{1}(\mathbb R^{n})$ and $b\in L^{1}(\mathbb R^{n}),$  then the  Hardy-Littlewood Maximal Operator M 
\begin{align*}
M :W(L^{r}, L^{q})(\mathbb R^{n})\rightarrow W(L_{a}^{p)}, L_{b}^{q)})(\mathbb R^{n})
\end{align*}
is bounded
\end{proposition}
\begin{proof}
It is known  that the  Hardy-Littlewood Maximal Operator M
\begin{align*}
M:L^p(\mathbb R^{n})\rightarrow L^p(\mathbb R^{n})
\end{align*}
is bounded for $1<p<\infty$ ( see Theorem 1, in \cite {st} ).Thus there exists a constant $C_1>0$ such that 
\begin{align}
\|Mh\|_{L^p}\leq C_1\|h\|_{L^p} 
\end{align}
for all $h\in  L^p(\mathbb R^{n}).$ Since $a\in L^{1}(\mathbb R^{n})$ and $b\in L^{1}(\mathbb R^{n}),$  then  we write
\begin{align}
L^{p}(\mathbb R^{n})\hookrightarrow L_{a}^{p)}(\mathbb R^{n}) , L^{q}(\mathbb R^{n})\hookrightarrow L_{a}^{q)}(\Omega). 
\end{align}
Since $p\leq q,$ by the properties of Wiener amalgam space and from $(4.2)$ we observe that
\begin{align*}
L^{p}(\mathbb R^{n})=W(L^{p}(\mathbb R^{n}), L^{p}(\mathbb R^{n}))\hookrightarrow W(L^{p}(\mathbb R^{n}), L^q(\mathbb R^{n}))\hookrightarrow W(L_{a}^{p)}(\mathbb R^{n}), L_{b}^{q)}(\mathbb R^{n})).
\end{align*}
 Hence the unit map
\begin{align*}
I :L^{p}(\mathbb R^{n})\rightarrow W(L_{a}^{p)}(\mathbb R^{n}), L_{b}^{q)}(\mathbb R^{n}))
\end{align*}
is bounded. So, there exists $C_2>0$ such that for all $g\in L^{p}(\mathbb R^{n}),$
\begin{align}
\|I(g)\|_{W(L_{a}^{p)}, L_{b}^{q)})}=\|g\|_{W(L_{a}^{p)}, L_{b}^{q)})}\leq C_2\|g\|_{L^p}. 
\end{align}
Let  $f\in L^{p}(\mathbb R^{n}).$ Since
 \begin{align*}
L^{p}(\mathbb R^{n})\hookrightarrow W(L_{a}^{p)}(\mathbb R^{n}), L_{b}^{q)}(\mathbb R^{n})),
\end{align*}
by (4.1), (4.3) and the hypothesis $p\leq q,$ we have
\begin{align}
\|Mf\|_{W(L_{a}^{p)}, L_{b}^{q)})}\leq C_2\|Mf\|_{L^p}\leq C_1 C_2\|f\|_{L^p}. 
\end{align}
Thus the  Hardy-Littlewood Maximal Operator M, 
\begin{align*}
M:L^{p}(\mathbb R^{n})\rightarrow W(L_{a}^{p)}(\mathbb R^{n}), L_{b}^{q)}(\mathbb R^{n}))
\end{align*}
is bounded. Since  $p\leq r,$  by the properties of Wiener amalgam space we have the embedding 
\begin{align}
W(L^{r}(\mathbb R^{n}), L^{p}(\mathbb R^{n}))\hookrightarrow W(L^{p}(\mathbb R^{n}), L^{p}(\mathbb R^{n}))=L^{p}(\mathbb R^{n}).
\end{align}
Then by (4.4) and (4.5),
\begin{align}
\|Mf\|_{W(L_{a}^{p)}, L_{b}^{q)})}\leq C_1 C_2\|f\|_{L^p}\leq C_1 C_2 C_3 \|f\|_{W(L^{r}, L^{p})}
\end{align}
for some $C _3>0$.
 Finally from (5.6)  the  Hardy-Littlewood Maximal Operator M, 
\begin{align*}
M :W(L_{a}^{r)}(\mathbb R^{n}), L_{b}^{p)}(\mathbb R^{n}))\rightarrow W(L_{a}^{p)}(\mathbb R^{n}), L_{b}^{q)}(\mathbb R^{n}))
\end{align*}
is bounded.

In \cite{u}, the weighted generalized grand Lebesgue space $L_{a}^{p)}(\Omega,a \omega)$  is defined by taking a second weight function $w(x)$ next to $a(x)$ by the norm
\begin{align*}
\notag\left\Vert {f}\right\Vert _{L^{p)}_a(\Omega,\omega)}& = \sup_{0<\varepsilon \leq p-1}\left(
\varepsilon \int_{\Omega }\left\vert f\right\vert ^{p-\varepsilon} a(x)^{\varepsilon}\omega(x) dx
\right) ^{\frac{1}{p-\varepsilon}
}\\&=\sup_{0<\varepsilon \leq p-1}\varepsilon ^{\frac{1 }{p-\varepsilon }}\left\Vert {f}\right\Vert
_{L^{p-\varepsilon }(\Omega, a^{\varepsilon}(x)\omega(x))}.
\end{align*}
\end{proof}
It is easy to see that $L_{a}^{p)}(\Omega,a \omega)$  is a kind of generalization of the spaces generalized grand Lebesgue space $L_{a}^{p)}(\Omega)$ and  weighted Lebesgue space $ L^{p}(\Omega,\omega).$ For that reason  the  grand  Wiener amalgam space $ W(L_{a}^{p)}(\mathbb R^{n}), L^{q}(\mathbb R^{n},\omega))$ is defined as in Definition2.1.

\begin{proposition}
 Let $a(x)$ and $w(x)$ be weight functions. Assume that $a(x)$  is  Beurling's weight,  $\frac{1}{\omega^{\frac{1}{q}}}\in L^{r}(\mathbb R^{n})$ and  $\frac{1}{q}+\frac{1}{r}=1. $ Then the Hardy-Littlewood maximal operator M,
\begin{align*}
M :W(L_{a}^{p)}(\mathbb R^{n}),L^{q}(\mathbb R^{n},\omega))\rightarrow W(L_{a}^{p)}(\mathbb R^{n}),L^{q}(\mathbb R^{n},\omega))
\end{align*}
is not bounded.
\end{proposition}
\begin{proof}
Let $ g\in L^{q}(\mathbb R^{n},\omega).$ Then $g.\omega^{\frac{1}{q}}\in L^q(\mathbb R^{n}).$ Since $\frac{1}{\omega^{\frac{1}{q}}}\in L^r(\mathbb R^{n})$ and   $\frac{1}{q}+\frac{1}{r}=1, $ by the Holder's inequality $g\in L^1(\mathbb R^{n}).$ Thus $ L^{q}(\mathbb R^{n},\omega)\subset L^1(\mathbb R^{n}).$ Since   $a(x)>1,$ then 
\begin{align}
\notag\left\Vert {f}\right\Vert _{L^{p)}_a(\mathbb R^{n})}& = \sup_{0<\varepsilon \leq p-1}\left(
\varepsilon \int_{\mathbb R^{n} }\left\vert f\right\vert ^{p-\varepsilon} a(x)^{\varepsilon} dx
\right) ^{\frac{1}{p-\varepsilon}
}\\&=\sup_{0<\varepsilon \leq p-1}\varepsilon ^{\frac{1 }{p-\varepsilon }}\left\Vert {f}\right\Vert
\notag_{L^{p-\varepsilon }(\mathbb R^{n}, a^{\varepsilon})}\\&\ge\sup_{0<\varepsilon \leq p-1}\varepsilon ^{\frac{1 }{p-\varepsilon }}\left\Vert {f}\right\Vert_{L^{p-\varepsilon }(\mathbb R^{n})}.
\end{align}
Take a fixed $0<\varepsilon_0\leq p-1.$ Hence from (4.7) we write 
\begin{align}
\left\Vert {f}\right\Vert _{L^{p)}_a(\mathbb R^{n})}\ge\varepsilon_0 ^{\frac{1 }{p-\varepsilon_0 }}\left\Vert {f}\right\Vert
\notag_{L^{p-\varepsilon_0 }(\mathbb R^{n}, a^{\varepsilon_0})}\ge\varepsilon_0 ^{\frac{1 }{p-\varepsilon_0 }}\left\Vert {f}\right\Vert_{L^{p-\varepsilon_0 }(\mathbb R^{n})} 
\end{align}
for a fixed $0<\varepsilon_0<p-1.$ Thus we have the inclusions $L^{p)}_a(\mathbb R^{n})\subset L^{p-\varepsilon_0}(\mathbb R^{n}, a^{\frac{\varepsilon_0}p})\subset L^{p-\varepsilon_0}(\mathbb R^{n}).$ This implies the following nesting property:
\begin{align}
\notag W(L_{a}^{p)}(\mathbb R^{n}),L^{q}(\mathbb R^{n},\omega))\subset W(L^{p-\varepsilon_0}(\mathbb R^{n}, a^{\frac{\varepsilon_0}p}), L^{q}(\mathbb R^{n},\omega))\subset  W(L^{p-\varepsilon_0}(\mathbb R^{n}), L^{1}(\mathbb R^{n}))&\\\subset W( L^{1}(\mathbb R^{n}),  L^{1}(\mathbb R^{n}) =L^{1}(\mathbb R^{n})). 
\end{align}
It is known by Theorem 1, in \cite {st}  that the Hardy-Littlewood maximal operator $ M,$ 
\begin{align}
M:  L^{1}(\mathbb R^{n})\rightarrow  L^{1}(\mathbb R^{n})
\end{align}
is not bounded. Let $E$ be a compact subset of $\mathbb R^{n}.$ Take the characteristic function $\chi_ {E}$. It is easy to show that  $\chi_ {E}\in L^{1}(\mathbb R^{n}).$  Then we obtain  $M(\chi_ {E})\notin  L^{1}(\mathbb R^{n}).$  This implies from (4.9) that $M(\chi_ {E})\notin W(L_{a}^{p)}(\mathbb R^{n}), L_{\omega}^{q}(\mathbb R^{n})).$  This completes the proof.
\end{proof}

\end{document}